\def\draft{0}
\def\done{1}

\if \draft1
\documentclass[draft,11pt]{amsart}
\else
\documentclass[11pt]{amsart}
\fi
\usepackage{amsmath, amsthm, amsfonts, amsrefs, amssymb}
\usepackage{colonequals}
\usepackage{aliascnt}
\usepackage[colorlinks=true,linkcolor=red!40!black,citecolor=green!40!black,urlcolor=cyan!40!black]{hyperref}
\usepackage{tikz}
\usetikzlibrary{intersections}
\usepackage[matrix, arrow]{xy}

\usepackage{verbatim}

\if \draft1
\usepackage[color]{showkeys}
\fi

\if \done0
\usepackage{todonotes}
\usepackage{everypage}
\usepackage[absolute]{textpos}
\AddEverypageHook{
\begin{textblock}{1}(0,1.5)
\centering
\color{red}
Draft: \today
\end{textblock}
}
\fi

\title{A Remark on Roberts' Totally Twisted Khovanov Homology}
\author{Thomas C. Jaeger}
\email{tcjaeger@syr.edu}
\address{Department of Mathematics;
215 Carnegie Building;
Syracuse University;
Syracuse, NY 13244-1150}

\newtheorem{theorem}{Theorem}[section]

\newcommand\newthm[2]{
\newaliascnt{#1}{theorem}
\newtheorem{#1}[#1]{#2}
\aliascntresetthe{#1}
\expandafter\newcommand\csname #1autorefname\endcsname{#2}
}

\newthm{lemma}{Lemma}
\newthm{proposition}{Proposition}
\newthm{corollary}{Corollary}
\theoremstyle{definition}
\newthm{definition}{Definition}
\newthm{remark}{Remark}

\tikzstyle{upper} = [preaction={draw=white, -, line width=4pt}]

\DeclareMathOperator{\Hom}{Hom}
\newcommand{\id}{\mathrm{id}}
\newcommand{\rKh}{\overline{Kh}}

\newcommand\F{\mathbb F}

\newcommand\Kob{\mathbf{Kob}}

\newcommand\CC{\mathcal C}

\newcommand\ce{\colonequals}

\newcommand\tikzcmd[2]{
\expandafter\newsavebox\csname #1savebox\endcsname
\expandafter\savebox\csname #1savebox\endcsname{\begin{tikzpicture}#2\end{tikzpicture}}
\expandafter\newcommand\csname #1\endcsname{\expandafter\usebox\csname #1savebox\endcsname}
}

\tikzcmd{vi}{[scale=0.15,baseline=-0.1cm]
\draw (-1,-1) to[out=45,in=-90] (-0.5,0) to[out=90,in=-45] (-1,1);
\draw (1,-1) to[out=135,in=-90] (0.5,0) to[out=90,in=-135] (1,1);
}

\tikzcmd{vx}{[scale=0.15,baseline=-0.1cm]
\draw (-1,-1) to[out=45,in=-90] (-0.5,0) to[out=90,in=-45] (-1,1);
\draw (1,-1) to[out=135,in=-90] (0.5,0) to[out=90,in=-135] (1,1);
\draw[fill] (-0.5,0) circle (0.2);
}

\tikzcmd{vxy}{[scale=0.15,baseline=-0.1cm]
\draw (-1,-1) to[out=45,in=-90] (-0.5,0) to[out=90,in=-45] (-1,1);
\draw (1,-1) to[out=135,in=-90] (0.5,0) to[out=90,in=-135] (1,1);
\draw[fill] (-0.5,0) circle (0.2);
\draw[fill] (0.5,0) circle (0.2);
}

\tikzcmd{vh}{[scale=0.15,baseline=-0.1cm]
\draw (-1,-1) to[out=45,in=-90] (-0.5,0) to[out=90,in=-45] (-1,1);
\draw (1,-1) to[out=135,in=-90] (0.5,0) to[out=90,in=-135] (1,1);
\draw (-0.5,0)--(0.5,0);
}

\tikzcmd{vhx}{[scale=0.15,baseline=-0.1cm]
\draw (-1,-1) to[out=45,in=-90] (-0.5,0) to[out=90,in=-45] (-1,1);
\draw (1,-1) to[out=135,in=-90] (0.5,0) to[out=90,in=-135] (1,1);
\draw (-0.5,0)--(0.5,0);
\draw[fill] (0,0) circle (0.2);
}

\newcommand\vy{\reflectbox\vx}
\newcommand\hi{\raisebox{-0.05cm}{\rotatebox{90}\vi}}
\newcommand\hx{\raisebox{-0.05cm}{\rotatebox{90}\vx}}
\newcommand\hy{\raisebox{-0.05cm}{\rotatebox{90}\vy}}

\newcommand\hv{\raisebox{-0.05cm}{\rotatebox{90}\vh}}

\tikzcmd{xpos}{[scale=0.15, baseline=-0.1cm]
\draw (135:-1.4) -- (135:-0.4);
\draw[->] (135:0.4) -- (135:1.4);
\draw[->] ( 45:-1.4) -- ( 45:1.4);
}

\tikzcmd{xneg}{[scale=0.15, baseline=-0.1cm]
\draw (45:-1.4) -- (45:-0.4);
\draw[->] (45:0.4) -- (45:1.4);
\draw[->] (135:-1.4) -- (135:1.4);
}

\tikzcmd{mtopleft}{[scale=0.15, baseline=-0.1cm]
\draw (45:-1.4) -- (45:-0.4);
\draw (45:0.4) -- (45:1.4);
\draw (135:-1.4) -- (135:1.4);
\begin{scope}[rotate=45]
\draw (-0.4,0.75)--(0.4,0.75);
\end{scope}
}

\tikzcmd{mtopright}{[scale=0.15, baseline=-0.1cm]
\draw (45:-1.4) -- (45:-0.4);
\draw (45:0.4) -- (45:1.4);
\draw (135:-1.4) -- (135:1.4);
\begin{scope}[rotate=-45]
\draw (-0.4,0.75)--(0.4,0.75);
\end{scope}
}

\tikzcmd{mbottomleft}{[scale=0.15, baseline=-0.1cm]
\draw (45:-1.4) -- (45:-0.4);
\draw (45:0.4) -- (45:1.4);
\draw (135:-1.4) -- (135:1.4);
\begin{scope}[rotate=135]
\draw (-0.4,0.75)--(0.4,0.75);
\end{scope}
}

\tikzcmd{mbottomright}{[scale=0.15, baseline=-0.1cm]
\draw (45:-1.4) -- (45:-0.4);
\draw (45:0.4) -- (45:1.4);
\draw (135:-1.4) -- (135:1.4);
\begin{scope}[rotate=-135]
\draw (-0.4,0.75)--(0.4,0.75);
\end{scope}
}

\tikzcmd{xx}{[scale=0.15, baseline=-0.1cm]
\draw (0,-1) -- (0.7,-0.3);
\draw (1.3,0.3) to[out=45,in=135] (2.7,0.3);
\draw (3.3,-0.3) -- (4,-1);
\draw (0,1) -- (1.3,-0.3) to[out=-45,in=-135] (2.7,-0.3) -- (4,1);
}

\tikzcmd{mvv}{[scale=0.15, baseline=-0.1cm]
\draw (0,-1) -- (0,1);
\draw (-0.4,0) -- (0.4,0);
}

\tikzcmd{mmvv}{[scale=0.15, baseline=-0.1cm]
\draw (0,-1) -- (0,1);
\draw (-0.4,0.33) -- (0.4,0.33);
\draw (-0.4,-0.33) -- (0.4,-0.33);
}

\tikzcmd{vv}{[scale=0.15, baseline=-0.1cm]
\draw (0,-1) -- (0,1);
}

\tikzcmd{xvv}{[scale=0.15, baseline=-0.1cm]
\draw (0,-1) -- (0,1);
\draw[fill] (0,0) circle (0.2);
}

\tikzcmd{vhs}{[scale=0.4]
\coordinate (A0) at (-1.35, -0.7);
\coordinate (B0) at (0.65, -0.7);
\coordinate (C0) at (1.35, 0.7);
\coordinate (D0) at (-0.65, 0.7);
\coordinate (A1) at (-1.35, 1.3);
\coordinate (B1) at (0.65, 1.3);
\coordinate (C1) at (1.35, 2.7);
\coordinate (D1) at (-0.65, 2.7);
\coordinate (X0) at (-0.45,0.2);
\coordinate (Y0) at (0.45,-0.2);
\coordinate (X1) at (-0.175,1.65);
\coordinate (Y1) at (0.175,2.35);
\coordinate (O) at (0,1);

\coordinate (Z) at (0.65, 0.3);

\draw (A0)--(A1);
\draw (B0)--(B1);
\draw (C0)--(C1);
\path[name path=p1] (D0)--(D1);
\draw[name path=p2] (A1) to[out=27.4,in=180] (X1) to[out=0,in=132.9] (B1);
\draw (D1) to[out=-47.1,in=180] (Y1) to[out=0,in=-152.6] (C1);
\draw (A0) to[out=27.4,in=-90] (X0);
\draw[densely dotted] (X0) to[out=90,in=-47.1] (D0);
\draw (B0) to[out=132.9,in=-90] (Y0);
\draw[densely dotted] (Y0) to[out=90,in=-132.9] (Z);
\draw (Z) to[out=47.1,in=-152.6,looseness=0.75] (C0);
\draw (X0) to[out=90,in=180,looseness=0.75] (O) to[in=90,out=0,looseness=0.5] (Y0);
\draw[densely dotted] (X1) to[out=-90,in=180,looseness=0.5] (O) to[in=-90,out=0,looseness=0.3] (Y1);

\draw[name intersections={of=p1 and p2, by=Dh}] (Dh)--(D1);
\draw[densely dotted] (D0)--(Dh);
}

\tikzcmd{hvs}{[scale=0.4,rotate=180]
\coordinate (A0) at (-1.35, -0.7);
\coordinate (B0) at (0.65, -0.7);
\coordinate (C0) at (1.35, 0.7);
\coordinate (D0) at (-0.65, 0.7);
\coordinate (A1) at (-1.35, 1.3);
\coordinate (B1) at (0.65, 1.3);
\coordinate (C1) at (1.35, 2.7);
\coordinate (D1) at (-0.65, 2.7);
\coordinate (X0) at (-0.45,0.2);
\coordinate (Y0) at (0.45,-0.2);
\coordinate (X1) at (-0.175,1.65);
\coordinate (Y1) at (0.175,2.35);
\coordinate (O) at (0,1);

\coordinate (Z) at (-0.65, 1.6);

\draw (A0)--(A1);
\path[name path=p1] (B0)--(B1);
\draw (C0)--(C1);
\draw (D0)--(D1);
\draw (A1) to[out=27.4,in=-165] (Z);
\draw[densely dotted] (Z) to[out=15,in=180] (X1) to[out=0,in=132.9] (B1);
\draw (D1) to[out=-47.1,in=180] (Y1) to[out=0,in=-152.6] (C1);
\draw (A0) to[out=27.4,in=-90] (X0) to[out=90,in=-47.1] (D0);
\draw (B0) to[out=132.9,in=-90] (Y0);
\draw[name path=p2] (Y0) to[out=90,in=-152.6,looseness=0.75] (C0);
\draw (X0) to[out=90,in=180,looseness=0.75] (O) to[in=90,out=0,looseness=0.5] (Y0);
\draw[densely dotted] (X1) to[out=-90,in=180,looseness=0.5] (O) to[in=-90,out=0,looseness=0.3] (Y1);

\draw[name intersections={of=p1 and p2, by=Bh}] (Bh)--(B0);
\draw[densely dotted] (B1)--(Bh);
}

\allowdisplaybreaks

\begin{document}

\begin{abstract}
We offer an alternative construction of Roberts' totally twisted Khovanov
homology and prove that it agrees with $\delta$-graded reduced
characteristic-$2$ Khovanov homology.
\end{abstract}

\maketitle

\section{Introduction}
Spanning tree models have been successfully used to study both knot Floer
homology and Khovanov homology of (quasi-)alternating links.  They appear
to occur in two flavors:  On the one hand, the models described by
Ozsv\'ath and Szab\'o for knot Floer homology \cite{ozsvath-szabo} and
by Wehrli~\cite{wehrli} and Champanerkar and Kofman~\cite{champanerkar-kofman}
for Khovanov homology completely determine the homology of the knot, but
their differentials are in general difficult to compute.  On the other hand,
the spanning tree models recently introduced by Baldwin and Levine for
knot Floer homology \cite{baldwin-levine} and by Roberts for Khovanov homology
\cite{roberts} have completely explicit differentials, but only recover
the homology with the bigrading collapsed to a single $\delta$-grading.

In the present note, we answer a question raised by Roberts, showing that for
knots, his theory indeed coincides with $\delta$-graded reduced Khovanov
homology.

\section{Definitions}

In order to be able to express the theory locally, we will need to relax the notion of a chain complex:

\begin{definition}
If $\CC$ is a graded additive category whose $\Hom$-sets are $\F_2$-algebras,
then we say that a \emph{chain complex} over $\CC$ is a an object $C$ of $\CC$
together with a differential $d: C \to C$ of degree $2$ such that $d^2 = 0$.
\end{definition}

When we deal with multiple gradings, we will refer to this grading as the
\emph{homological grading}. Based on this definition, it is straightforward to
redefine basic notions of homological algebra: A \emph{chain morphism} between
two chain complexes is a degree-$0$ morphism between the underlying objects
that commutes with the differentials and a \emph{chain homotopy} between chain
morphisms $f: C \to C'$ and $g: C \to C'$ is a morphism $h: C' \to C$ of degree
$-2$ such that $f + g = dh + hd$.  Given a tensor product $\otimes: \CC_1
\times \CC_2 \to \CC$, we define the tensor product of chain complexes $C_1 \in
\CC_1$ and $C_2 \in \CC_2$ as the object $C = C_1 \otimes C_2$ with
differential $d = d_1 \otimes 1 + 1 \otimes d_2$.  It is easy to see that this
operation is associative, but note that this uses the assumption that we work
in characteristic $2$ in an essential way.  In the familiar case that the
category $\CC$ is the category of graded modules over an (ungraded) ring $R$, the
above notions clearly agree with the usual definitions.

For our definition of Khovanov homology, we follow Bar-Natan's approach (see
\cite{bar-natan} for the general ideas and \cite{bar-natan-fast} for a
modification that works in characteristic $2$).  Let $R$ be a ring of
characteristic $2$.  The Khovanov homology of a tangle with endpoints $E$ is a
chain complex (in the above sense) over a certain bigraded category $\Kob(E)$
of crossingless tangles and $R$-linear combinations of dotted cobordisms
subject to certain relations modeled after the Frobenius algebra structure of
$R[x]/(x^2)$.  Specifically, the complexes associated to the positive and
negative crossings are
\[
Kh(\xpos): \raisebox{-0.1cm}{\scalebox2\vi}\{1\} \xrightarrow{\vhs} \raisebox{-0.1cm}{\scalebox2\hi}[2]\{2\}
\]
and
\[
Kh(\xneg): \raisebox{-0.1cm}{\scalebox2\hi}[-2]\{-2\} \xrightarrow\hvs \raisebox{-0.1cm}{\scalebox2\vi}\{-1\}
\text.
\]
We write $\vh$ for the differential of $Kh(\xpos)$ and $\hv$ for the
differential of $Kh(\xneg)$.
Here $[\cdot]$ denotes a shift in homological grading and $\{\cdot\}$ a shift
in $q$-grading (which is defined using the Euler characteristic of the
cobordisms).  We use the following conventions: The $q$-grading of a dot is
$-2$; if $f: C \to C'$ is a morphism of $q$-degree $d$, then the corresponding
morphism $f: C\{n\} \to C\{m\}$ has $q$-degree $d + m - n$ and similarly for
homological degree.

The reduced Khovanov homology of a link $L$ can now be defined as follows:
$Kh(L)$ is computed by tensoring the Khovanov complexes of the individual
crossings.  $Kh(L)$ is a complex of free $R[x]$-modules, where $x$ is the
action of the dot at the basepoint.  Here we used implicitly that
$\Kob(\emptyset)$ is equivalent to the category of finitely generated free
$R$-modules. The reduced Khovanov complex is simply $\rKh(L) \ce Kh(L)
\otimes_{R[x]} R$, where $x$ acts trivially on $R$.

We now define a twisted version of Khovanov homology.  Let $p_1, \dotsc, p_n$
be a collection of markings on the edges of a link diagram and $w_i \in R$ be
weights associated to those markings.  In addition to crossings, we add another
type of elementary diagram: an edge with a single marking of weight $w \in R$,
which we depict as $\mvv \, w$ (and often just as $\mvv$ if the weight is clear
from context).  We define $Kh(\mvv \, w) \ce \vv$ with
differential $w \, \xvv$.  The reduced twisted Khovanov homology associated
to a given marking can now be defined by taking the tensor product corresponding
to gluing the crossings and markings.  See \autoref{example-hopf} for an example.

Notice that the differential of the crossings has homological degree $2$ and
$q$-degree $0$, whereas the differential of markings has homological degree $0$
and $q$-degree $-2$.  In order for the definition to make sense, we therefore
need to collapse the two gradings to a single $\delta$-grading, where (with the
conventions used above) the $\delta$-grading is the difference of homological and
$q$-grading (note that our homological grading is twice the grading one usual
considers in Khovanov homology).

\begin{figure}
\hspace{-8cm}
\begin{tikzpicture}
\draw (0,0) to[out=90,in=180] (1,1);
\draw (0,0) to[out=-90,in=-180] (1,-1);
\draw (2,1) to[out=0,in=90] (3,0);
\draw (2,-1) to[out=-0,in=-90] (3,0);
\draw (1,1) to[out=-0,in=90] (2,0);
\draw[upper] (2,1) to[out=-180,in=90] (1,0);
\draw (1,0) to[out=-90,in=180] (2,-1);
\draw[upper] (2,0) to[out=-90,in=0] (1,-1);
\fill (0,0) circle (0.07cm);
\draw (0.9,0)--(1.1,0) node[right]{$w_1$};
\draw (1.9,0)--(2.1,0) node[right]{$w_2$};
\draw (2.9,0)--(3.1,0) node[right]{$w_3$};
\end{tikzpicture}
\vspace{-2.5cm}

\newcommand\hopfresolution[2]{
\begin{tikzpicture}[scale=0.5, baseline=-0.05cm]
\draw (0,0) to[out=90,in=180] (1,1);
\draw (0,0) to[out=-90,in=-180] (1,-1);
\draw (2,1) to[out=0,in=90] (3,0);
\draw (2,-1) to[out=-0,in=-90] (3,0);
\if 0#1
\draw (1,1) to[out=-0,in=90] (1,0);
\draw (2,1) to[out=-180,in=90] (2,0);
\else
\draw (1,1) to[out=-0,in=-180] (2,1);
\draw (1,0) to[out=90,in=90] (2,0);
\fi
\if 0#2
\draw (2,0) to[out=-90,in=180] (2,-1);
\draw (1,0) to[out=-90,in=0] (1,-1);
\else
\draw (1,0) to[out=-90,in=-90] (2,0);
\draw (2,-1) to[out=180,in=0] (1,-1);
\fi
\fill (0,0) circle (0.1cm);
\draw (0.8,0)--(1.2,0);
\draw (1.8,0)--(2.2,0);
\draw (2.8,0)--(3.2,0);
\end{tikzpicture}
}

\[
\xymatrix@R=0.1cm@C=1.5cm{
& & \bullet \ar[dddd]^{w_1+w_2 \quad \hopfresolution11} \\
& \hopfresolution10 \\ \\
& \bullet \ar[dr]^1 \\
\bullet \ar[dddd]_{\hopfresolution00 \quad w_2+w_3} \ar[ur]^1 \ar[dr]_1 & & \bullet \\
& \bullet \ar[ur]_1 \\ \\
& \hopfresolution01 \\
\bullet
}
\]
\caption{Twisted Khovanov homology of the Hopf link. Basepoint and weights are
as indicated at the top left of the figure. The rest of the figure shows the
differential of the complex, where each dot represents a copy of $R$. The
complex can still be viewed as a cube of resolutions; the four resolutions of
the link are drawn near their corresponding subquotient complexes.  Notice that
the horizontal differential is identical to the differential of the reduced
Khovanov homology of the Hopf link.}
\label{example-hopf}
\end{figure}
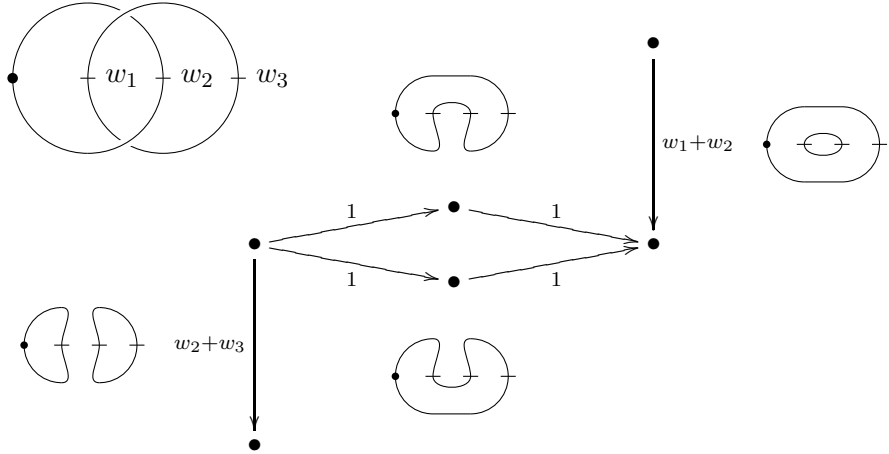

\section{Properties of twisted Khovanov homology}

In this section we prove the following invariance result.
\begin{theorem}\label{invariance}
Let $D$ be the diagram of a tangle with markings of weight
$w_{i1}, \dotsc, w_{in_i}$ on the $i$th component of the tangle. Then the
isomorphism type $Kh(D)$ depends only on $w_i \ce \sum_{j=1}^{n_i} w_{ij}$.
\end{theorem}

\begin{corollary}
If $K$ is a knot then $\rKh(K)$ does not depend on the markings, in
particular, it is isomorphic to $\delta$-graded Khovanov homology.
\end{corollary}
\begin{proof}
By the theorem above (technically applied to the knot cut open at the
basepoint), we may move all markings to the same edge that the basepoint lies
on.  Thus the differential of $Kh(K)$ is $d + w x$, where $d$ is the
differential of the untwisted Khovanov homology, $w$ is the sum of all the
weights, and $x$ is the action of the dot at the basepoint as before. Hence the
differential on $\rKh(K)$ is simply $d \otimes 1$.
\end{proof}

\begin{proof}[of \autoref{invariance}]
First, note that the Khovanov complex of two adjacent markings of weight $w_1$
and $w_2$ is $Kh(\mmvv) = \vv$ with differential
$d = w_1 \, \xvv \otimes 1 + 1 \otimes w_2 \, \xvv = (w_1 + w_2) \, \xvv$.

Next, we prove that we can slide a marking past a crossing:
\[
\xymatrix@R=1.5cm@C=3cm{
\vi \oplus \hi
\ar[r]^{\begin{pmatrix} w \, \vx \\ \vh & w \, \hy \end{pmatrix}}
\ar@<0.3ex>[d]_{\begin{pmatrix} \vi & w \hv \\ & \hi \end{pmatrix}}
&
\vi \oplus \hi
\ar@<0.3ex>[d]_{\begin{pmatrix} \vi & w \hv \\ & \hi \end{pmatrix}}
\\
\vi \oplus \hi
\ar[r]_{\begin{pmatrix} w \, \vy \\ \vh & w \, \hx \end{pmatrix}}
\ar@<0.3ex>[u]_{\begin{pmatrix} \vi & w \hv \\ & \hi \end{pmatrix}}
&
\vi \oplus \hi
\ar@<0.3ex>[u]_{\begin{pmatrix} \vi & w \hv \\ & \hi \end{pmatrix}}
}
\]
The top row is the differential of $Kh(\mtopleft)$ and the bottom row
is the differential of $Kh(\mbottomright)$.  The arrows pointing up
and down define an isomorphism between $Kh(\mtopleft)$ and
$Kh(\mbottomright)$, as can be easily checked.  An isomorphism between
$Kh(\mtopright)$ and $Kh(\mbottomleft)$ can be established in a similar fashion.

Finally, we note that since we can always slide markings away from tangles
without closed components, invariance of the homotopy type of $Kh(L)$ under
Reidemeister moves follows immediately from the invariance of the untwisted
theory.

\end{proof}

\section{Spanning Trees}

The version of twisted Khovanov homology defined in this note admits
a spanning tree model along the same lines as the models in \cite{baldwin-levine}
and \cite{roberts}.
$\rKh(L)$ has a natural double complex structure: The horizontal differential
is simply the differential of untwisted reduced Khovanov homology, whereas
the vertical differential is induced by the weights.  This double complex
structure gives rise to a spectral sequence converging to $\rKh(L)$.

If $R$ is a field and each edge of the graph (with the exception of the edge that
contains the basepoint) has a non-zero weight associated to it and all weights
are linearly independent over $\F_2$, we can describe the spectral sequence
explicitly: The $d_0$ differential cancels all states corresponding to
disconnected resolutions of the diagram and the remaining states are in
one-to-one correspondence with spanning trees.  For grading reasons, all
subsequent differentials except for the $d_2$ differential are zero and an
explicit calculation shows that the $d_2$ differential between two resolutions
$R$ and $R'$ that differ at two places (such that $R$ is the $0$-resolution at
those places and $R'$ is the $1$-resolution) is
$\frac1{\sum_i w_i} + \frac1{\sum_i w'_i}$, where $w_i$ and $w'_i$ are the
weights on the circles not containing the basepoint in the two intermediate
resolutions between $R$ and $R'$.  Otherwise, $d_2 = 0$.

\section{Relation to Roberts' totally twisted Khovanov homology}

In this section we show that the theory described earlier coincides
with Roberts' totally twisted Khovanov homology.  Robert's theory
is defined over the ring $\F_2(x_1, \dotsc, x_n)$, where the $x_i$
are labels associated to the regions of the knot diagram.  The variables
associated to the two regions $x_{n-1}$ and $x_{n}$ adjacent to the basepoint
are never used in the differential, hence we may consider the same theory
over the ring $R' \ce \F_2(x_1, \dotsc, x_{n-2})$.
Define vector spaces $X \ce \F_2 \langle x_1, \dotsc, x_{n-2} \rangle$ and
$Y \ce \F_2 \langle y_1, \dotsc, y_m \rangle$, where the $y_j$ are variables
associated to the edges of the diagram not containing the basepoint and define
a linear map $f: X \to Y$ by sending $f(x_i)$ to the sum of the $y_j$
associated to the edges that make up the boundary of the region labeled $x_i$.  We
claim that $f$ is injective.  Assume that $f(\sum_i \lambda_i x_i) = 0$.  If
$x_{i_1}$ and $x_{i_2}$ ($i_1, i_2 \leq n-2$) are the labels of two adjacent
regions, then $\lambda_{i_1} = \lambda_{i_2}$ since the $y_j$ associated to
their common edge does not appear in any of the $f(x_i)$ for $x \neq x_{i_1},
x_{i_2}$.  Thus $\lambda_1 = \dotso = \lambda_{n-2}$.  By considering an edge
adjacent to $x_{n-1}$ or $x_n$, we see that in fact $\lambda_i = 0$ for all
$1 \leq i \leq n-2$, hence $f$ is injective as claimed.  Fix a linear map
$g: Y \to X$ satisfying $g \circ f = \id_X$.  Note that if $C$ is a
basepoint-avoiding circle in any resolution of the diagram, then the sum over
all $x_i$ corresponding to the region of $S^2 \backslash C$ not containing the basepoint
is the same as the sum over $g(y_j)$, where $y_j$ runs over the edges of $C$.
Thus we can obtain Robert's theory by working over the ring $R'$ and assigning
the weight $w_j \ce g(y_j)$ to a marking on the $j$th edge.

Similarly, given weights $w_j \in R$, we obtain the theory described in this
note from Roberts' definitions by specializing $x_i$ to be the sum of the
weights of the markings on the boundary of the region labeled $x_i$.

\end{document}